\documentclass[11pt]{amsart}
\usepackage{setspace}
\usepackage[english]{babel}
\usepackage[T1]{fontenc}
\usepackage{geometry}

\usepackage{fullpage}
\usepackage[dvips]{graphicx}
\usepackage{color}

\usepackage{amsmath,amscd,amsthm,amsfonts,latexsym,epsfig,multirow}

\theoremstyle{plain}
\usepackage{amssymb}
\usepackage{newlfont}
\usepackage{amsfonts}
\usepackage{graphicx}
\usepackage{xfrac}
\usepackage{enumerate}
\usepackage{tikz}
\usepackage[all]{xy}
\usepackage{subfigure}

\newtheorem{theorem}{Theorem}
\newtheorem{proposition}{Proposition}

\theoremstyle{definition}
\newtheorem{definition}{Definition}
\newtheorem{remark}{Remark}

\newtheorem*{acknowledgements}{Acknowledgements}

\newtheorem{thmy}{Theorem}

\def \r{\mbox{${\mathbb R}$}}
\renewcommand{\Re}{\mathcal Re}
\begin{document}

\title{A family of higher genus complete minimal surfaces that includes the Costa-Hoffman-Meeks one}

\author{Irene I. Onnis}
\address{Universit\`a degli Studi di Cagliari\\
Dipartimento di Matematica e Informatica\\
Via Ospedale 72\\
09124 Cagliari, Italy.}
\email{irene.onnis@unica.it}

\author{B\'arbara C. Val\'erio}
\address{Instituto de Matem\'atica e Estat\'istica\\
              Universidade de S\~ao Paulo\\
              05508-090 S\~ao Paulo, SP, Brazil. }
\email{barbarav@ime.usp.br}

\author{Jos\'{e} Antonio M. Vilhena}
\address{Universidade Federal do Par\'a\\ 
Instituto de Ci\^encias Exatas e Naturais\\ Faculdade de Matem\'atica\\ Rua Augusto Corr\^ea 01, 66075-110 Bel\'{e}m, PA,  Brazil.}
\email{vilhena@ufpa.br}

\date{March 2023}
\subjclass{53A10; 53C42}
\keywords{Examples of minimal surfaces. Elliptic functions. Riemann surface. Planar-type end. Catenoid-type end. Enneper-type end. }
\thanks{The first author was supported by 
 GNSAGA-INdAM,  by a grant of FdS (Projects GOACT and ISI-HOMOS), by the Thematic Project: Topologia Álgebrica,  Geométrica e Diferencial, Fapesp process number 2022/16455-6 and partially funded by PNRR e.INS Ecosystem of Innovation for Next Generation Sardinia. }

\begin{abstract}
In this paper, we construct a one-parameter family of minimal surfaces in the Euclidean $3$-space of arbitrarily high genus and with three ends.   Each member of this family is immersed, complete and with finite total curvature.  Another interesting property is that the symmetry group of the genus $k$ surfaces $\Sigma_{k,t}$ is the dihedral group with $4(k+1)$ elements.  Moreover, in particular,  for $|t|=1$ we find the family of the Costa-Hoffman-Meeks embedded minimal surfaces,  which have two catenoidal ends and a middle flat end.  Among the non-embedded examples obtained,  there are noncongruent minimal surfaces,  with the same symmetry group and conformal structure, as we have in \cite{Batista.2004}.
\end{abstract}
\maketitle

\section{Introduction}

The Enneper surface is, after the plane, the simplest complete minimal surface in the three-dimensional Euclidean space.  It has genus zero and total  Gauss curvature equal to $-4\pi$.  In \cite{Chen.1982} C.C.  Chen and F.   Gackstatter constructed the first examples of complete minimal surfaces of genus one and two,  with one Enneper-type end of winding order three,  the same symmetries as the Enneper surface and total curvature $-8\pi$ and $-12\pi$, respectively.   In 1994 N. Do Espirito-Santo (\cite{Nedir.1994}) showed the existence of a complete minimal surface immersed in $\r^3$ of genus three,  having one Enneper-type end and total curvature $-16\pi$.  In \cite{Karcher.1989} H.  Karcher generalized the genus one Chen-Gackstatter surface by increasing the genus and the winding order of the end.  
 A similar generalization of the genus two Chen-Gackstatter surface was obtained by C.E.  Thayer in \cite{Thayer.1995},  solving numerically the associated period problem for genus as high as $35$.  Other generalizations of these Enneper-type minimal surfaces are given in \cite{Fang.1990} and \cite{Kang.2003}.\\

For what concerns minimal surfaces of genus one,  in 1982 C.  Costa has constructed a complete minimal surface of genus one in $\mathbb{R}^3$ with three embedded ends (see \cite{Costa.1982,  Costa.1984}).  Later,   D.  Hoffman and W.H.  Meeks in  \cite{Hoffman.1985}  have proved its embeddedness. In  \cite{Hoffman.1990}, they extended those results to higher genus surfaces, proving the existence of an infinite family  of embedded minimal surfaces $M_k$ of genus $k \geq 1$, with two catenoidal ends and one planar middle end (see the Main Theorem,  p. ~$1$). The total curvature of $M_k$ is $-4\pi(k+2)$ and its symmetry group is the dihedral group $\mathcal{D}(2k+2)$ with $4(k+1)$ elements generated by the orthogonal transformations
\begin{equation}\label{matrix}
	K=\left[
	\begin{array}{ccc}
	1&0&0\\
	0&-1&0\\
	0&0&1
	\end{array}
	\right]\qquad\text{and}\qquad
L_{\theta} = \left[
	\begin{array}{ccc}
	\mathcal{R}_{\theta}& &0\\
	 & &0\\
	0&0&-1
	\end{array}
	\right],
\end{equation}
where $\mathcal{R}_{\theta}$ represents the matrix of a rotation by $\theta=\pi/(k+1)$ radians in the $(x_1,x_2)$-plane.  The surface $M_k$ is known as ``Costa-Hoffman-Meeks surface'' of genus $k$.  \\

Many interesting new examples of minimal surfaces have been obtained by combining the examples above. In \cite{Batista.2004},  V.  Ramos  Batista has introduced minimal immersions of punctured compact Riemann surfaces in the Euclidean $3$-space  called ``{\it birdcage-catenoids}'',  whose construction is indirectly based on the Costa surface.  In fact,  we can think of the ``birdcage'' of genus $k-1$, with $k\geq 3$,  as $M_k$ with each end compactified to a point.  Via these examples the author has showed  that the symmetry group and the conformal structure together of minimal surfaces in $\r^3$, with finite total curvature and positive genus, are insufficient to characterize the surfaces. \\

In the beautiful work \cite{Batista.2006}  from F.   Martín and V.  Ramos Bastista, the Costa and Scherk’s minimal surfaces have  been combined to obtain new example of embedded singly periodic minimal surfaces,  called {\it Scherk-Costa surfaces} (SC-surfaces, in short). 
The construction of the fundamental piece of such surfaces follows by replacing each
end of the Costa surface by Scherk-type ends, preserving the straight lines crossing at the ``Costa-saddle'.  By successive $180^\circ$-rotations on the straight lines, one obtains the whole singly periodic SC-surface.\\

Inspired by the afore mentioned results, Section 3 of this paper uses the theory of elliptic functions and the Enneper-Weierstrass Representation Theorem to prove the existence of a one-parameter family of minimal surfaces of genus one immersed in $\mathbb{R}^3$, having the Costa surface as a distinguished member. The other ones are obtained replacing its catenoidal ends by Enneper type ends. Furthermore, all those surfaces $S^1_x$ have the same symmetries as the Costa surface.   Namely, part of Theorem~\ref{teo3} may be restated as follows:

\begin{thmy}\label{TeoA}
\textit{There exists a one-parameter family of complete, genus one,   
  minimal surfaces which are immersed in $\mathbb{R}^3$,  with  three ends and finite total curvature.  The family depends on a parameter $x$ with $|x| < \sqrt{\sqrt{8}-1}\, e_1$,  where $e_1=\wp(1/2)$ and $\wp$ is the Weierstrass function.  Moreover,  if $|x|=e_1$ one gets the Costa surface and if $|x|\neq e_1$,  the corresponding surfaces have total curvature equal to $-20 \pi$, two Enneper-type ends and one middle planar end.}
  \end{thmy}

For any $x$ in the above range, we denote by $S^1_x$ the corresponding element in the family. In  Figure~\ref{Fig-sup} we picture four of their elements, highlighting the fundamental piece and symmetries. This, as well as all the pictures presented here, were rendered with Maple software, a trademark of Waterloo Maple Inc.\\

	\begin{figure}[h!]\label{Fig-sup}
\subfigure[\label{Figure_a}]{
\includegraphics[totalheight=5.5cm]{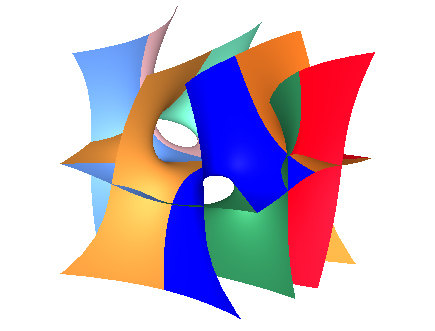}}
\subfigure[\label{Figure_b}]{
\includegraphics[totalheight=5.5cm]{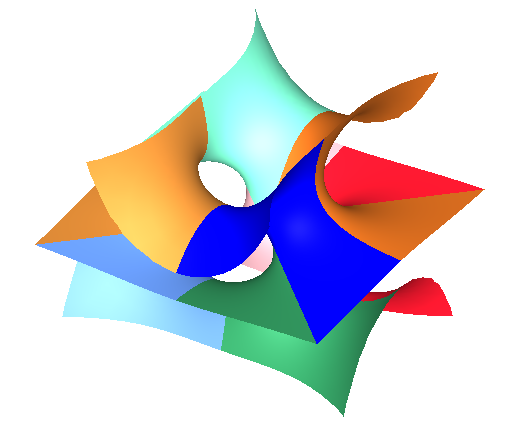}}
\subfigure[\label{Figure_c}]{
\includegraphics[totalheight=5.5cm]{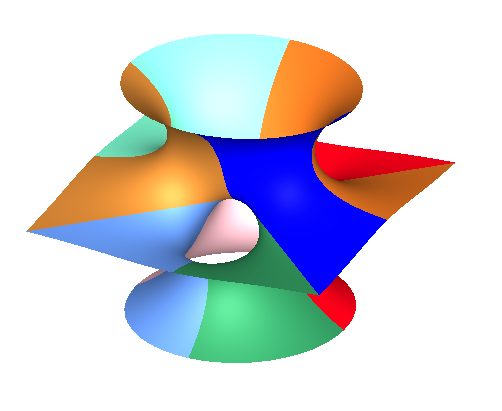}}
\subfigure[\label{Figure_d}]{
\includegraphics[totalheight=5.5cm]{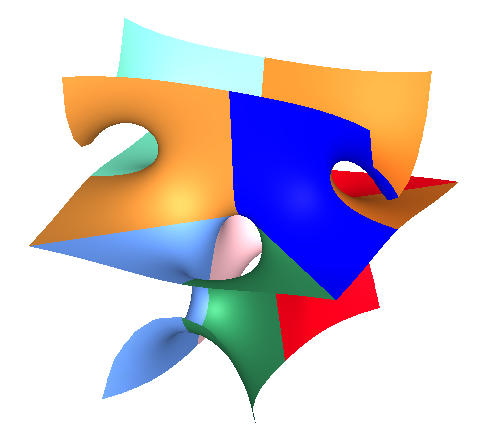}}
\caption{Computer graphics of the genus one minimal surfaces $S_x^1$ for: (a) $x=0$,  (b) $x=-\sfrac{1}{2}+e_1$,  (c) $x=e_1$ (Costa surface) and (d) $x=\sfrac{1}{2}+e_1.$}
\end{figure}

We point out that in the proof of Theorem~\ref{teo3} we use the Weierstrass $\wp$-function following the classical approach given by C.~Costa, C.C.~Chen and F. ~Gackstatter.  Our construction mostly resorts to undergraduate complex analysis \cite{Ahlfors.1966}, whence accessible for readers from many research areas. We also remark that would be possible to prove the Theorem~\ref{teo3} by using the powerful reverse construction method given by H.~Karcher in \cite{Karcher.1989}.  \\

Section~\ref{Sec4} is dedicated to describe a larger family of immersed and complete minimal surfaces $\Sigma_{k,t}$ that generalizes the family $S_x^1$ above, by increasing the genus of the elements from $1$ to any integer $k \geq 1$. Such surfaces are deformations of the Costa-Hoffman-Meeks surfaces of genus $k$, where the two catenoidal ends are replaced by Enneper type ends, while preserving the original symmetries. This deformation occurs by bending $(k + 1)$ brims of each catenoidal end towards the planar end,
intersecting it transversely. In the same catenoid end, the remaining $(k+1)$ brims are bended outwards the planar end. All that brims pass through the symmetry planes of the surface. The deformation of the other
catenoidal end respects symmetries imposed by the straight lines and planar geodesics in the surface $\Sigma_{k,t}$.  As we shall see, in this notation, we have
\begin{itemize}
\item $\Sigma_{1,1}$ is the Costa surface;
\item  $\Sigma_{1,t}$ are the genus one minimal surfaces $S_x^1$ of Theorem~\ref{TeoA};
\item the sub-family $\Sigma_{k,t}$ for $|t|=1$,  reduces to the family of Costa-Hoffman-Meeks  embedded minimal surfaces $M_k$.  
\end{itemize}

More precisely, the main goal of this paper is to prove the following
\begin{thmy}
\textit{There exists a one-parameter family  $\Sigma_{k,t}$ of complete minimal surfaces in $\mathbb{R}^3$ of genus $k$,  with finite total curvature and three ends,  containing as sub-family $\Sigma_{k,\pm 1}$ the Costa-Hoffman-Meeks family  of embedded  surfaces $M_k$. Moreover,  if $|t| \neq 1$ and $|t| < \sqrt{2\sqrt{k+1}-1},$ then the immersed minimal surfaces $\Sigma_{k,t}$ have total curvature $ -4\pi(3k+2)$,  two Enneper-type ends and one middle planar end.}
\end{thmy}
We emphasize that, for each $k$ and $t$, $\Sigma_{k,t}$ and Costa-Hoffman-Meeks surface of genus $k$ share the same group of symmetries. Hence each one of them may be decomposed in $4(k+1)$ congruent fundamental pieces. Figure~\ref{Fig-g-2} depicts four elements of the family $\Sigma_{2,t}$, built from the action of the twelve elements in the symmetry group of the surface on one of its fundamental pieces. In the pictures of the entire surface in Figure~\ref{Fig-g-2} the corresponding $12$ pieces are detached using different colors and were produced using the software Maple via the Weierstrass data  $(g, \eta)$. 
\begin{figure}[h!]
\subfigure[\label{Figure_a2}]{
\includegraphics[height=4.1cm]{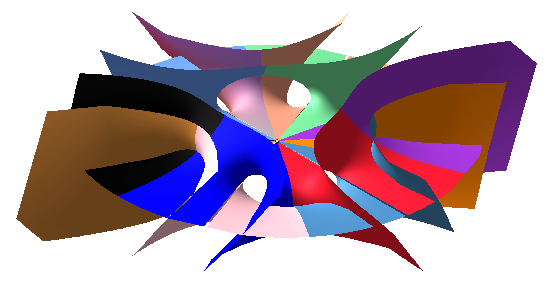}}
\subfigure[\label{Figure_b2}]{
\includegraphics[height=4.5cm]{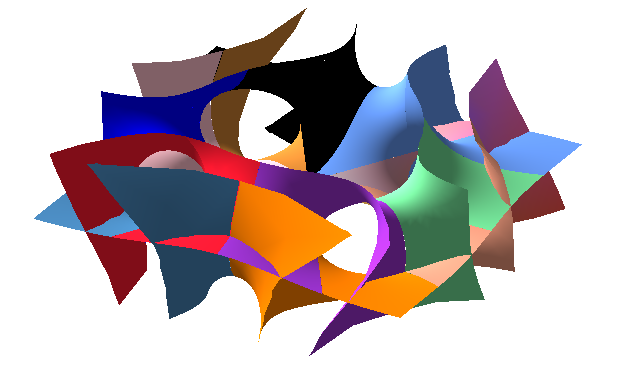}}
\subfigure[\label{Figure_c2}]{
\includegraphics[height=4cm]{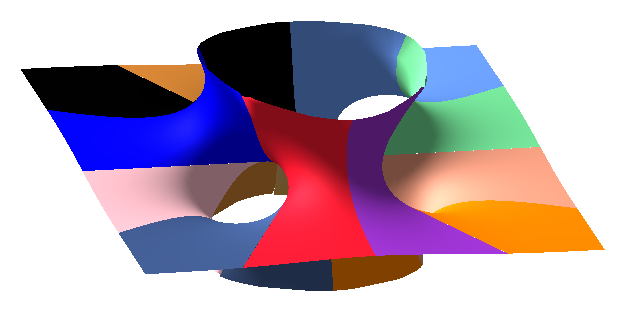}}
\subfigure[\label{Figure_d2}]{
\includegraphics[height=4.1cm]{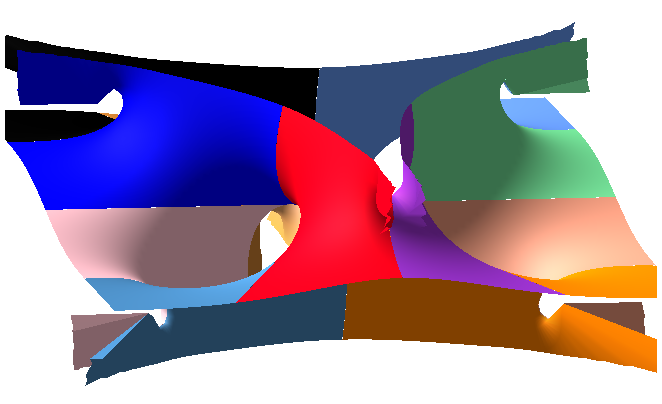}}
\caption[Computer graphics of the genus two minimal surfaces $\Sigma_{2,t}$ obtained for: (a) $t=0$, (b) $t=0.8$, (c) $t=1$ (Costa-Hoffman-Meeks surface) and (d) $t=1.2.$]{Computer graphics\footnotemark \,of the genus two minimal surfaces $\Sigma_{2,t}$ obtained for: (a) $t=0$, (b) $t=0.8$, (c) $t=1$ (Costa-Hoffman-Meeks surface) and (d) $t=1.2.$}
\label{Fig-g-2}
\end{figure}
\footnotetext{According to the expression of $c(k,t)$ in \eqref{const_c},  the estimations of the parameters that solve the period problem are given by  $c(2,0)=1.158615284$,  $c(2,0.8)=1.065926062$, $c(2,1)=0.9880704345$ and $c(2,1.2)=0.8589964085$,  respectively.}
\\

The results presented in Section~\ref{sec3}, which focuses on surfaces of genus $1$, are of intrinsic interest, even though they are generalized to higher
genus examples in Section~\ref{Sec4}. In the genus $1$ case, the techniques used in the construction of such surfaces and the study of their symmetries highlight the properties of Weierstrass elliptic function $\wp$, recorded in Section~\ref{sec2}. Moreover, it is important to notice that both $1$-parameter families of minimal surfaces obtained in Theorem~\ref{teo3} provide noncongruent examples, immersed in $\r^3$, with the same symmetries and conformal structure (see Figure~\ref{Fig4ab}), as we have in \cite{Batista.2004}.  This rises the question whether uniqueness could be a consequence of the following setup: given two complete embedded minimal surfaces of finite total curvature in a flat space,  if they have the same symmetries and conformal structure, do these two surfaces coincide?\\

Eventually,  this question can motivate some readers to try to construct triply periodic minimal surfaces based on the examples in \cite{Batista.2004} by replacing the catenoidal ends with planar curves of symmetry. Six possible constructions can be carried out, namely from birdcage-catenoids of genera two, three and five. Since the birdcage-catenoids don’t imply uniqueness, then whether such examples exist the aforementioned question admits a negative answer. 

\section{Basic theory and Enneper-Weierstrass representation}\label{sec2}
This section is devoted to state some important results about elliptic functions and minimal surfaces,  widely used throughout this work. The reader can find the details in \cite {Chand.1985},  \cite{Costa.1984},  \cite{Fang.1990},  \cite{Osserman.2013} and \cite{vilhena.2021}.   We start reviewing one of the principal tools used to construct minimal surfaces in the Euclidean $3$-dimensional space,  in the formulation due to Osserman \cite{Osserman.2013}:

\begin{theorem}[Enneper-Weierstrass representation]\label{W} 
Let $\overline{M}$ be a compact Riemann surface and $M =\overline{M}-\{p_1, \cdots, p_n\}.$ Suppose that $\mathfrak{g}: \overline{M} \rightarrow \mathbb{C}\cup \{\infty\} $ is a meromorphic function and $\eta$ is a meromorphic $1$-form such that whenever $g=\mathfrak{g}|_M$ has a pole of order $k$, then $\eta$ has a zero of order $2k$ and $\eta$ has no other zeros on $M$. Let
\begin{equation}\label{WR}
\phi_1=\left(1-g^2 \right)\eta, \quad  \phi_2=i\left(1+g^2 \right)\eta, \quad \phi_3=2g\, \eta. 
\end{equation}
If for any closed curve $\alpha$ in $M$,
\begin{equation}\label{PP1}
\Re\int_{\alpha} \phi_j=0, \qquad  j =1,2,3,
\end{equation}
and every divergent curve $\gamma$ in $M$ has infinite length, i.e.,
\begin{equation}\label{PS0}
\int_{\gamma} (1+|g|^2)\,|\eta|=\infty,
\end{equation}
then the surface $S$, defined by $X: M \rightarrow \mathbb{R}^3$, is a complete regular minimal surface, where
\begin{equation}\label{PS}
X(z)=\Re\left(\int_{z_0}^z \phi_1, \int_{z_0}^z \phi_2, \int_{z_0}^z \phi_3\right).
\end{equation}
Here, $z_0$ is any fixed point of $M$. Moreover, 
the total curvature of $S$ is
\begin{equation}\label{CT}
C_T(S) = -4\pi\, \mathrm{deg}(\mathfrak{g}).
\end{equation}
\end{theorem}
\begin{definition}
Given the lattice  
$\mathcal{L} = [1,i]$ in the complex plane $\mathbb{C}$,  the Weierstrass $\wp$ function (relative to $\mathcal{L} $) is the doubly periodic meromorphic function defined by 
\begin{equation}\label{PW}
\wp(z)=\frac{1}{z^2}+\sum_{\substack{\omega \in \mathcal{L}, \,\omega \neq 0}}\bigg\{\frac{1}{(z-\omega)^2}-\frac{1}{\omega^2} \bigg\}.
\end{equation}
The Weierstrass $\zeta$ function is defined by
\begin{equation}\label{ZW}
  \zeta(z)=\frac{1}{z}+\sum_{\substack{\omega \in \mathcal{L}, \,\omega \neq 0}}
  \left(\frac{1}{z-\omega}+\frac{1}{\omega}+\frac{z}{\omega^2}\right) 
\end{equation}
and 
\begin{equation*}\label{ZP}
\zeta'(z)=-\wp(z).
\end{equation*}
\end{definition}
Let $F=\{ z \in \mathbb{C} \ | \ 0 \leqslant \Re(z)  < 1, \   0 \leqslant {\mathcal Im}(z) < 1\}$  be a fundamental domain  of $\wp$. Then, in $F$ the elliptic function $\wp(z)$ satisfies  the differential equation
\begin{equation}\label{EDPW}
	(\wp')^2= 4\wp^3 - g_2 \wp,  
\end{equation}
where  $g_2=60\!\!\!\!\!\!\sum\limits_{\omega \in \mathcal{L}, \,\omega \neq 0}{\omega^{-4}}.$
Setting $e_1=\wp(1/2)$ we have $\displaystyle{e_1=\frac{\sqrt{g_2}}{2}  \approx 6.875185815}$.  Also,  by further differentiation of \eqref{EDPW}, we obtain
\begin{equation}\label{EP2}
\wp^2=\frac{\wp''}{6}+\frac{e_1^2}{3}.
\end{equation}
From the Addition Theorem (see, for example,  \cite{Chand.1985}) it's easy to prove the following
\begin{proposition}\label{prop2}
Let $\mathcal{L} = [1,i]$ and $z \in F$. Then, we have:
\begin{equation}\label{pe1}
\frac{2e_1^2}{\wp-e_1}=\wp(z-1/2)-e_1,
\qquad
\frac{2e_1^2}{\wp+e_1}=\wp(z-i/2)+e_1.
\end{equation}
Therefore, from \eqref{EP2},  we find that
\begin{equation}\label{pe1bis}
\frac{24 e_1^4}{(\wp-e_1)^2}=\wp''(z-1/2)-12e_1\wp(z-1/2)+8e_1^2,
\end{equation}
\begin{equation}\label{pe3bis}
\frac{24 e_1^4}{(\wp+e_1)^2}=\wp''(z-i/2)+12e_1\wp(z-i/2)+8e_1^2.
\end{equation}
\end{proposition}
By using the Legendre's relation we have 
\begin{proposition}[\cite{Costa.1984}]\label{prop6}
Let $\alpha_i: [0,1] \rightarrow \mathbb{C}$, $i=1,2$, be the paths
\begin{equation*}
\alpha_1(t) =\frac{i}{3}+t, \qquad \alpha_2(t)=\frac{1}{3}+it
\end{equation*}
of the homology basis of the torus $\displaystyle{T^2=\mathbb{C}/\mathcal{L}}.$ Then,
\begin{equation}\label{intA}
\int_{\alpha_1} \wp(z) \, dz = -\pi, \qquad \int_{\alpha_2} \wp(z) \, dz = i\pi.
\end{equation}
\end{proposition}

The symmetries of the minimal surfaces of genus $1$ that we construct in Section~\ref{sec3}  are  a consequence of the symmetries of the Weierstrass $\wp$ function in its fundamental domain $F$ (see Proposition~\ref{Prop3}) and of Proposition~\ref{Karcher.1989} below.

\begin{proposition}[\cite{Hoffman.1985}]\label{Prop3}
Let $\wp(z)$ be the Weierstrass $\wp$-function for the unit-square lattice $\mathcal{L}$ and $w_2=(1+i)/2$,  then 
\begin{enumerate}
\item $\wp(\rho(w_2+z))=-\wp(w_2+z), \quad \rho(w_2+z)=w_2+iz,$
\item $\wp(\beta(w_2+z))=\overline{\wp(w_2+z)}, \quad \beta(w_2+z)=w_2+\overline{z},$
\item $\wp(\rho\circ\beta(w_2+z))=-\overline{\wp(w_2+z)},$
\item $\wp(\rho^2\circ\beta(w_2+z))=\overline{\wp(w_2+z)},$
\item $\wp(\mu(w_2+z))=-\overline{\wp(w_2+z)}, \ \ \mu(w_2+z)=w_2-i\overline{z}. $
\end{enumerate}
\end{proposition}

\begin{remark}
We point out that $\rho$, $\beta$, $\rho^2 \circ \beta$, $\rho \circ \beta$ and $\mu$ are,  respectively, a rotation by $\pi/2$ about $w_2$, a reflection about the horizontal line, a reflection about the vertical line, a reflection about the positive diagonal and a reflection about the negative diagonal through $w_2$.
\end{remark}

\begin{proposition}[\cite{Karcher.1989},  \cite{Wohlgemuth.1991}]\label{Karcher.1989}
Let $\tau$ be a curve in $M$ such that $g \circ \tau$ is contained either in a meridian or in the equator of $\mathbb{S}^2$, and $(g\eta)(\tau')$ is real or purely imaginary. Then $\tau$ is a geodesic on $M$. Moreover,  if $\tau$ is a geodesic on $M$,  we have that
\begin{itemize}
\item [i)] $\tau$ is a planar curve of symmetry if, and only if $(dg \cdot\eta)(\tau')\in\r$; 
\item [ii)] $\tau$ is a straight line  if,  and only if  $(dg \cdot\eta)(\tau')\in i\,\r$.
\end{itemize}
\end{proposition}

The following is another crucial result in the classical minimal surface theory.
\begin{theorem}[Schwarz Reflection Principle (see, for instance,  \cite{Fujimoto.2013, Karcher.1989})]
\noindent \begin{itemize}
\item [i)] If a minimal surface contains a segment of straight line $L$, then it is symmetric under
rotation by $\pi$ about $L$.
\item [ii)] If a nonplanar minimal surface contains a principal geodesic,  which is necessarily
a planar curve, then it is symmetric under reflection in the plane of that
curve.
\end{itemize}
\end{theorem}

\section{One-parameter family of minimal surfaces with genus one and three ends}\label{sec3}

Now, making use of the theory presented in the previous section,  we will provide a proof of the existence of a family of genus one complete minimal surfaces, immersed in $\mathbb{R}^3$, with finite total curvature and three ends. More precisely, we prove the following:
\begin{theorem}\label{teo3}
There exists a one-parameter family  $S_x$, $x \in \mathbb{R}$,  of complete minimal surfaces immersed in $\mathbb{R}^3$ of genus one, with finite total curvature and three ends. Moreover, this family contains two sub-families with the following properties:
	\begin{enumerate}
	\item  The surfaces $S^1_x=S_{(x,y(x))}$, where
	$$
	y(x)=x, \qquad  |x| < \sqrt{\sqrt{8}-1}\,e_1,
	$$
which have symmetries ruled by the dihedral group $\mathcal{D}(4)$ of order $8$, generated by reflections in the $(x_1,x_3)$ and $(x_2,x_3)$ planes and rotations of $\pi$ radians about the lines $x_1\pm x_2=x_3=0$.
Also,
\begin{itemize}
    \item [i)] If $|x|\neq e_1$,  then $S^1_x$ has total curvature  $-20 \pi$, two Enneper-type ends and one middle planar end;
    \item [ii)] If $|x| = e_1$,  then $S^1_x$ is precisely the Costa surface.
  \end{itemize}
\item The surfaces $S^2_x=S_{(x,y(x))}$,  where
	$$
 y(x)=-5e_1^2/x, \qquad x\neq 0,
	$$
whose symmetries are generated by the reflections about two vertical, and mutually orthogonal, planes. Additionally,
\begin{itemize}
    \item [i)] If $|x|\neq e_1$ and $|x|\neq 5e_1$,  then $S^2_x$ has total curvature  $-20 \pi$, two Enneper-type ends and one middle planar end;
    \item [ii)] If $|x| = e_1$ or $|x| = 5e_1$,  then $S^2_x$ has total curvature  $-16 \pi$, one catenoid-type end, one Enneper-type end and one middle planar end.
  \end{itemize}
	\end{enumerate}

 \end{theorem}

\begin{proof}
Let $T^2=\mathbb{C}/\mathcal{L}$ be the torus with complex structure induced by the canonical projection $\mathfrak{p}: \mathbb{C} \rightarrow T^2.$ 
We consider $M=T^2 -\{p_1, p_2,p_3\}$,  where
$p_1=\mathfrak{p}(1/2)$,  $p_2=\mathfrak{p}(0)$ and $p_3=\mathfrak{p}(i/2)$.
The Weierstrass data  $(g, \eta)$ is given by
\begin{equation}\label{WR-WW}
\left\{
\begin{aligned}
	 g &= c \, \frac{\wp'}{\wp\, (\wp +x)\,(\wp-y)},  \qquad c >0,\\
	\mathbf \eta&=\frac{\wp\, (\wp +x)^2\,(\wp-y)^2}{(\wp-e_1)^2\,(\wp+e_1)^2} \, dz.
	\end{aligned}
	\right.
\end{equation}

 Figure~\ref{Fig1} shows the zeros and poles of $g$, $\eta$ and $dh=g \eta$ on $F$ in the case $|x|\neq e_1$ and $|y|\neq e_1$.
\begin{figure}[h!]\label{Fig1}
\subfigure[Zeros and poles of $g$.]{
\centering
\includegraphics[totalheight=3.6cm]{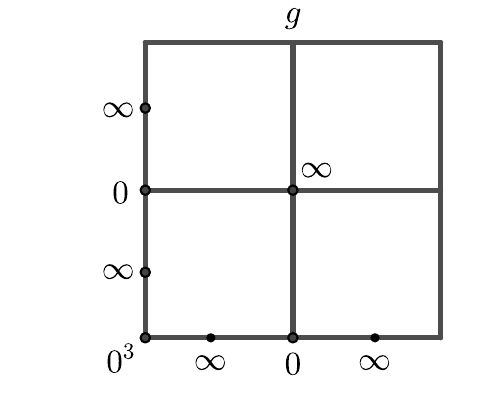}}
\subfigure[Zeros and poles of $\eta$.]{
\includegraphics[totalheight=3.6cm]{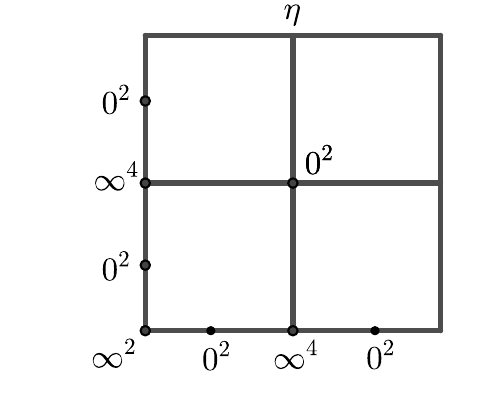}}
\subfigure[Zeros and poles of $dh$.]{
\includegraphics[totalheight=3.6cm]{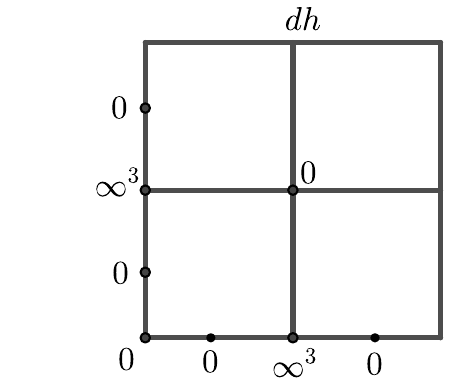}}
\caption{}
\end{figure}

The degree of the Gauss map $g$ can be equal to $3$ (if $y=x=\pm e_1$), $4$ (if $|x|=e_1$ or  $|y| = e_1$) and 
$5$ (if $|x|\neq e_1$ and $|y|\neq e_1$).  It results that the total curvature
\begin{equation}\label{CT2}
C_T(S_x):=\int_M K \, dA = -4 \pi \,\text{deg}(g),
\end{equation}
can be equal to $-12\pi$, $-16\pi$ and $-20\pi$, respectively.\\

By using \eqref{WR}, \eqref{EDPW} and \eqref{WR-WW},  we obtain
$$
\left\{
\begin{aligned}
\phi_1&=\frac{\wp\, (\wp +x)^2\,(\wp-y)^2}{(\wp-e_1)^2\,(\wp+e_1)^2} \, dz - \frac{4c^2}{(\wp-e_1)(\wp+e_1)}\,dz,\\
\phi_2&=\frac{i\wp\, (\wp +x)^2\,(\wp-y)^2}{(\wp-e_1)^2\,(\wp+e_1)^2} \, dz +  \frac{4 i \,c^2}{(\wp-e_1)(\wp+e_1)}\,dz,\\ 
      \phi_3&=2c\,\frac{\wp'\, (\wp +x)\,(\wp-y)}{(\wp-e_1)^2\,(\wp+e_1)^2} \, dz.
\end{aligned}
\right.
$$
Therefore,  decomposing in partial fractions,  we can write
$$
\left\{
\begin{aligned}
\phi_1&=\Big[\wp+ 2(x-y)+\frac{(e_1-x)(e_1+y)(2e_1+y-x)}{2e_1 (\wp+e_1)}+\frac{(e_1+x)(e_1-y)(2e_1-y+x)}{2e_1 (\wp-e_1)}\\& -\frac{(e_1-x)^2(e_1+y)^2}{4e_1 (\wp+e_1)^2}+\frac{(e_1+x)^2(e_1-y)^2}{4e_1 (\wp-e_1)^2}\Big]\, dz - \frac{2c^2}{e_1}\Big(\frac{1}{\wp-e_1}-\frac{1}{\wp+e_1}\Big) \,dz,\\  
\phi_2&=i\Big[\wp+ 2(x-y)+\frac{(e_1-x)(e_1+y)(2e_1+y-x)}{2e_1 (\wp+e_1)}+\frac{(e_1+x)(e_1-y)(2e_1-y+x)}{2e_1 (\wp-e_1)}\\& -\frac{(e_1-x)^2(e_1+y)^2}{4e_1 (\wp+e_1)^2}+\frac{(e_1+x)^2(e_1-y)^2}{4e_1 (\wp-e_1)^2}\Big]\, dz + \frac{2i\,c^2}{e_1}\Big(\frac{1}{\wp-e_1}-\frac{1}{\wp+e_1}\Big) \,dz,\\  
\phi_3&=\frac{c}{2e_1^3}\Big[(e_1^2+xy)\, \Big(\frac{\wp'}{\wp-e_1}-\frac{\wp'}{\wp+e_1}\Big)
+\frac{e_1\,(e_1+x)(e_1-y)\wp'}{(\wp-e_1)^2}+\frac{e_1\,(e_1-x)(e_1+y)\wp'}{(\wp+e_1)^2}\Big]\, dz.
\end{aligned}
\right.
$$
Now,  using the formulas of Proposition~\ref{prop2}, we get
\begin{equation*}
\begin{aligned}
\int \phi_1&=-\zeta(z)+ 2(x-y)z+\frac{(e_1-x)(e_1+y)(2e_1+y-x)\,[e_1 z-\zeta(z-i/2)]}{4e_1^3}\\&-\frac{(e_1+x)(e_1-y)(2e_1-y+x)[e_1 z+\zeta(z+1/2)]}{4e_1^3}\\& -\frac{(e_1-x)^2(e_1+y)^2}{4e_1}\,\Big[\frac{\wp'(z-i/2)}{24e_1^4}-\frac{\zeta(z-i/2)}{2e_1^3}+\frac{z}{3e_1^2}\Big]\\& +\frac{(e_1+x)^2(e_1-y)^2}{4e_1}\,\Big[\frac{\wp'(z-1/2)}{24e_1^4}+\frac{\zeta(z-1/2)}{2e_1^3}+\frac{z}{3e_1^2}\Big] \\&-c^2 \,\frac{\zeta(z-i/2)-\zeta(z-1/2)-2e_1z}{e_1^3},
\end{aligned}
\end{equation*}
\begin{equation*}
\begin{aligned}
\int \phi_2&=i\,\Big[-\zeta(z)+ 2(x-y)z+\frac{(e_1-x)(e_1+y)(2e_1+y-x)\,[e_1 z-\zeta(z-i/2)]}{4e_1^3}\\&-\frac{(e_1+x)(e_1-y)(2e_1-y+x)[e_1 z+\zeta(z+1/2)]}{4e_1^3}\\& -\frac{(e_1-x)^2(e_1+y)^2}{4e_1}\,\Big(\frac{\wp'(z-i/2)}{24e_1^4}-\frac{\zeta(z-i/2)}{2e_1^3}+\frac{z}{3e_1^2}\Big)\\& +\frac{(e_1+x)^2(e_1-y)^2}{4e_1}\,\Big(\frac{\wp'(z-1/2)}{24e_1^4}+\frac{\zeta(z-1/2)}{2e_1^3}+\frac{z}{3e_1^2}\Big)\Big] \\&+i\, c^2\, \frac{\zeta(z-i/2)-\zeta(z-1/2)-2e_1z}{e_1^3},
\end{aligned}
\end{equation*}
\begin{equation*}
\begin{aligned}
\int \phi_3&=\frac{c}{2e_1^3}\Big[(e_1^2+xy)\ln \Big(\frac{\wp-e_1}{\wp+e_1}\Big)-\frac{e_1(e_1+x)(e_1-y)}{\wp-e_1}-\frac{e_1(e_1-x)(e_1+y)}{\wp+e_1}\Big].
\end{aligned}
\end{equation*}\\

As $\wp$ is periodic,  $\int\phi_3$ is also periodic and then 
$$\int_{\alpha_i} \phi_3=0, \qquad i=1,2.$$
From Proposition~\ref{prop6}, it follows that
\begin{equation*}
\Re\int_{\alpha_2} \phi_1= \Re\int_{\alpha_1} \phi_2 =0
\end{equation*}
and
\begin{equation}\label{P12}
\int_{\alpha_1} \phi_1=F_1+\frac{2c^2}{e_1^2},\qquad
\int_{\alpha_2} \phi_2=F_2+\frac{2c^2}{e_1^2},
\end{equation}
where 
$$\left\{
\begin{aligned}
F_1&=\frac{10\, e_1^4\,(x-y)-21\pi e_1^4-3\pi e_1^2\, (x^2+y^2)+2e_1^2\, xy\,(x-y)+12\pi e_1^2\, xy+3\pi x^2y^2}{12\, e_1^4},\\
F_2&=\frac{10\, e_1^4\,(y-x)-21\pi e_1^4-3\pi e_1^2\, (x^2+y^2)+2e_1^2\, xy\,(y-x)+12\pi e_1^2\, xy+3\pi x^2y^2}{12\, e_1^4}.
\end{aligned}
\right.$$
Therefore \eqref{PP1} holds if, and only if $F_1+2c^2/e_1^2=0=F_2+2c^2/e_1^2$, that is
\begin{equation*}
(y-x)\,(xy+5e_1^2)=0,
\end{equation*}
and also
$$c(x,y)=\frac{e_1}{2}\sqrt{-(F_1+F_2)}=\frac{\sqrt{2\pi}}{4e_1}\sqrt{7e_1^4+(x^2+y^2-4xy)e_1^2-x^2y^2}.$$
Consequently, 
\begin{enumerate}
\item if $y=x$,  we get $$c(x)=\frac{\sqrt{2\pi}}{4e_1}\sqrt{8e_1^4-(x^2+e_1^2)^2}, \qquad x\in \Big(-\sqrt{\sqrt{8}-1}\,e_1,\sqrt{\sqrt{8}-1}\,e_1\Big);$$
\item if $y=-5e_1^2/x$,  we get $$c(x)=\frac{\sqrt{2\pi}}{4\,|x|}\sqrt{(x^2+e_1^2)^2+24e_1^4}, \qquad x\neq 0.$$
\end{enumerate}

The functions $\phi_1$, $\phi_2$ and $\phi_3$ have poles of order at least two at $p_1=\mathfrak{p}(1/2)$, $p_2=\mathfrak{p}(0)$ and $p_3=\mathfrak{p}(i/2)$. Hence, these functions have zero residues at $p_1, p_2$ and $p_3$, ensuring that the surfaces $S_x$ have no real periods around them. Therefore we obtain a one-parameter family of complete minimal surfaces $S_x$ containing two sub-families that will be indicated with $S^{n}_x$, $n=1,2$.  In  Figures~\ref{Fig-sup} and \ref{Sup_Cat_P_Enn} we present some pictures of minimal surfaces of the sub-families $S_x^1$ and $S_x^2$, respectively. 
\\

Finally,  in order to study the symmetries of the surfaces $S_x$,  we consider on $F$ the following curves:
\begin{equation}\label{eqcurva}
\begin{array}{lll}
&\tau_1(u)=u, \ 0<u<\sfrac{1}{2}, \qquad &\tau_2(u)=u, \ \sfrac{1}{2}<u<1,\\
&\tau_3(u)=\frac{i}{2}+u, \ 0<u<1, \qquad &\tau_4(u)=iu, \ 0<u<\sfrac{1}{2},\\
&\tau_5(u)=iu, \ \sfrac{1}{2}<u<1, \qquad &\tau_6(u)=\frac{1}{2}+iu, \ 0<u<1,\\
&\tau_7(u)=u+i(1-u), \ 0<u<1, \qquad &\tau_8(u)=u+iu, \ 0<u<1.
\end{array}
\end{equation}

Now one easily writes the expressions of the differential $dh=g \eta$ and $(dg\cdot \eta)$ as
$$\begin{aligned}
dh&=2c\,\sqrt{\frac{\wp}{(\wp^2-e_1^2)^3}}\, (\wp+x)(\wp-y)\, dz,\\
dg\cdot \eta&=\frac{2c\,[(x-y)\,(3\wp\,e_1^2-\wp^3)+\wp^2\,(5e_1^2-3\wp^2)-xy\, (\wp^2+e_1^2)]}{(\wp^2-e_1^2)^2}\,dz^2.
\end{aligned}$$
In Table~\ref{table1} we summarize the behavior of $g$, $dh$ and   $dg\cdot \eta$ along the path $\tau_j$, $j=1,\dots, 8$.
\begin{table}[h!]\caption{}
\renewcommand*{\arraystretch}{1.8}
\centering
\begin{tabular}{|c|c|c|c|c|}
 \hline
 \quad \text{Sub-family} $S^{n}_x$\quad  & \text{Path}\; $\tau_j$ & $g \circ \tau_j$  & $dh(\tau_j')$ & $(dg\cdot \eta)(\tau_j')$\\
 \hline
$n=1,2$ & $ j=1,2,3$  & $\r$    & $\r$ & $\r$ \\
 \hline
$n=1,2$ & $j=4,5,6$  & $\r$    & $i\r$& $\r$  \\
 \hline
 $n=1$ &  $j=7$ & $e^{\pm i\frac{\pi}{4}}\,\r$ & $i\r$ & $i\r$  \\ 
 \hline
$n=1$ & $j=8$ &$e^{\pm i\frac{\pi}{4}}\,\r$ &$\r$ &$i\r$  \\
 \hline
\end{tabular}
\label{table1}
\end{table}

From Proposition~\ref{Karcher.1989} and Table~\ref{table1}, we have that the curves
$\Gamma_j:=X\circ \tau_j$, $j=1,\ldots,6$, are planar geodesics of $S_x$. We can easily show that $\Gamma_j$, $j=1,2,3,$ are contained in a plane 
parallel to the $(x_1,x_3)$-plane,  $\Gamma_j$, $j=4,5,6,$ are contained in a plane parallel to the $(x_2,x_3)$-plane and $\Gamma_j, \ j=7,8,$ are contained in the lines $x_1\pm x_2=x_3=0$ of $S^1_x$. The Schwarz Reflection Principle for minimal surfaces implies that the surfaces $S^1_x$ have the $(x_1, x_3)$-plane and the $(x_2, x_3)$-plane as reflective planes of symmetry and, also, they are invariant under rotations by $\pi$ about the lines $x_1\pm x_2=x_3=0$.  Furthermore, the surfaces $S^2_x$ have only the $(x_1, x_3)$-plane and the $(x_2, x_3)$-plane as reflective planes of symmetry. \\

Finally,  to determine the symmetry group of the surfaces $S^1_x$,  we introduce the dihedral group $\mathcal{D}(4)$ of order eight which is generated by $\beta$ and $\rho$. As the generators $\beta$ and $\rho$ can be identified with the orthogonal motions given by
$$
K=\left[
	\begin{array}{ccc}
	1&0&0\\
	0&-1&0\\
	0&0&1
	\end{array}
	\right], \qquad
	L_{\pi/2}=\left[
	\begin{array}{ccc}
	0&-1&0\\
	1&0&0\\
	0&0&-1
	\end{array}
	\right],
$$
	it results that the group $\mathcal{D}(4)=\langle L_{\pi/2}^j \, K^\ell \rangle$, $ j=0, \ldots, 3$,  $\ell=0,1$, acts on $\r^3$.
Now,  from \eqref{PS}, \eqref{WR-WW} and Proposition~\ref{Prop3},  we have on $S^1_x$ the following properties:
$$\begin{aligned}
X(\beta(w_2+z))& =( X_1, X_2,X_3)(\beta(w_2+z))=(X_1,-X_2,X_3)(w_2+z),\\
X(\rho(w_2+z))&=( X_1, X_2,X_3)(\rho(w_2+z))=( -X_2, X_1,-X_3)(w_2+z).
\end{aligned}
$$
Therefore,  as $X \circ \beta = K\,  X^t$ and $X \circ \rho =  L_{\pi/2} \, X^t$, we conclude that the symmetry group of $S^1_x$ is $\mathcal{D}(4)$ and this completes the proof. 
\end{proof}

\begin{figure}[h!]
\subfigure[\label{Fig_a}]{
\includegraphics[totalheight=5.4cm]{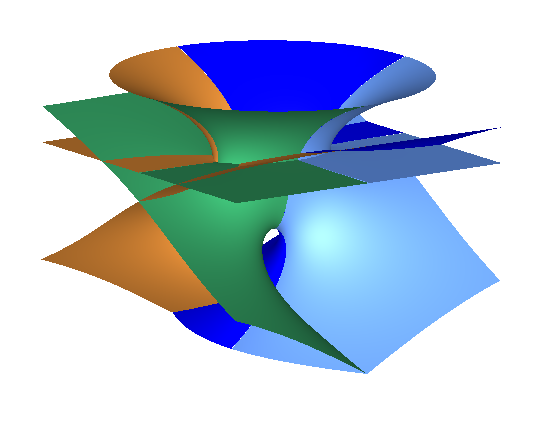}}
\subfigure[\label{Fig_b}]{
\includegraphics[totalheight=5.2cm]{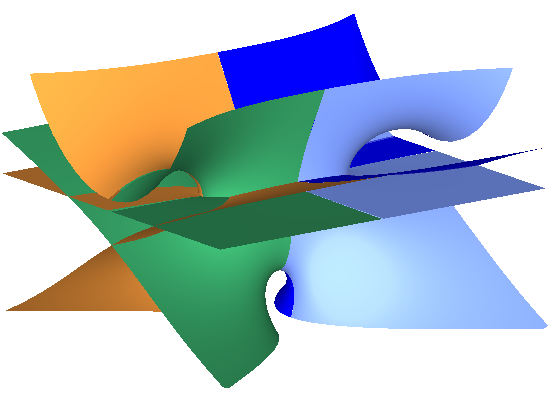}}
\caption{Computer graphics of the genus one minimal surfaces $S_x^2$ obtained for: (a)  $x=-e_1$ and (b)  $x=-e_1-0.5$.}\label{Sup_Cat_P_Enn}
\label{Fig4ab}
\end{figure}

\section{A family of high genus minimal surfaces with three ends}\label{Sec4}

The aim of this section is to find a one-parameter family  $\Sigma_{k,t}$ of complete minimal surfaces in $\mathbb{R}^3$ of genus $k$,  with finite total curvature and three ends, being one planar and two of Enneper-type, or one planar and two of catenoid-type.  Such family contains the one-parameter family of surfaces described in the Theorem~\ref{TeoA} and the Costa-Hoffman-Meeks embedded minimal surfaces $M_k$, of genus $k\geq 1$,  constructed in the Main Theorem of \cite{Hoffman.1990}.  The existence of this family is the main result of this paper,  summarized in the following theorem.

\begin{theorem}\label{Teo4}
For every integer $k \geq 1$, there exists a one-parameter family  $\Sigma_{k,t}$ of complete  
  minimal surfaces in $\mathbb{R}^3$ of genus $k$,  with finite total curvature and three ends.  Moreover such surfaces have the following properties:
  \begin{enumerate}
	\item If $|t| = 1$, then the minimal surfaces $\Sigma_{k,\pm 1}$ are precisely the Costa-Hoffman-Meeks embedded minimal surfaces  $M_k$.
	\item  If $|t| \neq 1$ and $|t| < \sqrt{2\sqrt{k+1}-1},$ then the immersed minimal surfaces $\Sigma_{k,t}$ have total curvature $C_T=-4\pi(3k+2)$,  two Enneper-type ends and one middle planar end.
	\item All the minimal surfaces $\Sigma_{k,t}$, $|t| < \sqrt{2\sqrt{k+1}-1},$ are symmetric by reflection about each of the $(k+1)$ vertical planes meeting in the $x_3$-axis and by rotation by $\pi$ radians about one of the $(k+1)$ straight lines on the surfaces in the $(x_1,x_2)$-plane.
	\item The symmetry group of $\Sigma_{k,t}$ is  the dihedral group $\mathcal{D}(2k+2)=\langle L_\theta^j \, K^\ell \rangle, \ j=0, \ldots , (2k+1)$,  $\ell=0,1$,  which has $4(k + 1)$ elements.
	\end{enumerate}
 \end{theorem}
\begin{proof}
For every integer $k \geq 1$,  we consider the closed Riemann surface of genus $k$ given by
\begin{equation}\label{RS}
 \overline{M}_k = \{(z,w) \in \mathbb{C}_\infty^2: w^{k+1} = z^k \, (z^2 - 1)\} .
\end{equation}
Let 
\begin{equation}
   \mathfrak{p}_0 = (0,0), \qquad \mathfrak{p}_{-1}=(-1,0), \qquad \mathfrak{p_1}=(1,0), \qquad \mathfrak{p}_{\infty}=(\infty, \infty)
\end{equation}
and we define $M_k=\overline{M}_k-\{\mathfrak{p}_{-1},\mathfrak{p}_1, \mathfrak{p}_{\infty}\}.$
 According to Enneper-Weierstrass Representation Theorem~\ref{W}, the Weierstrass data $(g,\eta)$ which we will use to construct the one-parameter family of minimal surfaces $\Sigma_{k,t}=X(M_k)$  is 
\begin{equation}\label{22}
\left\{
\begin{aligned}
	 g &=\Big( \frac{z^2-1}{z^2-t^2} \Big)\frac{c}{w},\\
	\mathbf \eta&=\frac{(z^2 - t^2)^2}{(z^2-1)^2} \left(\frac{z}{w}\right)^k dz,
	\end{aligned}
	\right.
\end{equation}
where $c\in\r$ is the unique positive constant for which we will prove that the immersion $X: M_k\to \r^3$ is well-defined on $M_k$.\\
 
In Table~\ref{table2} we summarize the behavior of zeros and poles of $g$, $\eta$ and $dh$ at the saddles $\mathfrak{p}_{-t}$,  $\mathfrak{p}_{0}$,  $\mathfrak{p}_{t}$, $t \neq \pm 1$, and at the ends $\mathfrak{p}_{-1}, \mathfrak{p}_1$  and $\mathfrak{p}_{\infty}$.
\begin{table}[h!]\caption{}
\renewcommand*{\arraystretch}{1.3}
\centering
\renewcommand{\tabcolsep}{3.6 pt}
\begin{tabular}{|c|c|c|c|c|c|c|c|}
 \hline
 \quad  $(z,w)$ & $\mathfrak{p}_{-1}$ & $\mathfrak{p}_{-t}$  &  $\mathfrak{p}_{0}$ & $\mathfrak{p}_{t}$ & $\mathfrak{p}_{1}$ & $\mathfrak{p}_{\infty}$\\
 \hline
 $g$ & $0^k$ & $\infty^{k+1}$  & $\infty^k$ & $\infty^{k+1}$ & $0^k$ & $0^{k+2}$\\
 \hline
 $\eta$ & $\infty^{2k+2}$ & $0^{2k+2}$  & $0^{2k}$ & $0^{2k+2}$ & $\infty^{2k+2}$ & $\infty^{2}$\\
 \hline
  $dh$ & $\infty^{k+2}$ & $0^{k+1}$  & $0^k$ & $0^{k+1}$ & $\infty^{k+2}$ & $0^{k}$\\
 \hline
 Ends/saddle & Enneper end  &  regular point  & regular point & regular point  &  Enneper end & planar end \\
 \hline
\end{tabular}
\label{table2}
\end{table} 

In this case the degree of $g$ on $\overline{M}_k$ is $\mathrm{deg}(\mathfrak{g})=3k+2$ and $C_T(S) = -4\pi\,(3k+2)$. From the Gackstatter \cite{Gackstatter.1976}, Jorge-Meeks \cite{Jorge.1983} formula given by
\begin{equation}\label{FJMG}
C_T(S)=2\pi \Bigg(2-2\mathbf{g}-N-\sum_{\nu=1}^N k_\nu\Bigg), 
\end{equation}
where $k_{\nu}$ is the order of the end,  it results that the order of the Enneper ends is $2k+1$ and the order of the flat end is $1$. Thus all the minimal surfaces $\Sigma_{k,t}$, $t \neq \pm 1$,  are not embedded. \\

If $t=\pm 1,$ then we reobtain the Costa-Hoffman-Meeks Weierstrass data
\begin{equation}\label{WR-HM}
\left\{
\begin{aligned}
	 g &= \frac{c}{w},\\
	\mathbf \eta&=\left(\frac{z}{w}\right)^k dz
	\end{aligned}
	\right.
\end{equation}
and from Table~\ref{table2} we also find the behavior of zeros and poles of $g$, $\eta$ and $dh$ at $\mathfrak{p}_{-1}$, $\mathfrak{p}_{0}$, $\mathfrak{p}_1$  and $\mathfrak{p}_{\infty}$,  see Table~\ref{table3} below. 

\begin{table}[h!]\caption{}
\renewcommand*{\arraystretch}{1.3}
\centering
\begin{tabular}{|c|c|c|c|c|}
 \hline
 \quad  $(z,w)$ & $\mathfrak{p}_{-1}$  &  $\mathfrak{p}_{0}$ & $\mathfrak{p}_{1}$ & $\mathfrak{p}_{\infty}$\\
 \hline
 $g$ & $\infty$ & $\infty^{k}$  & $\infty$ & $0^{k+2}$\\
 \hline
 $\eta$ & * & $0^{2k}$  & * & $\infty^{2}$\\
 \hline
  $dh$ & $\infty$ & $0^{k}$ & $\infty$ & $0^{k}$\\
 \hline
 Ends/saddle & catenoidal end  & regular point & catenoidal end  & planar end \\
 \hline
\end{tabular}
\label{table3}
\end{table} 

Making use of \eqref{22} we may write
$$\left\{\begin{aligned}
\phi_1&=\eta-c^2\,\eta_2,\\
\phi_2&=i\,(\eta+c^2\,\eta_2), \\
\phi_3&=\frac{c}{2}\left(\frac{1-t^2}{(z-1)^2}+ \frac{1-t^2}{(z+1)^2}-\frac{1+t^2}{z+1}+\frac{1+t^2}{z-1} \right)\, dz,
\end{aligned}
\right.$$
where 
$$\eta=\frac{(z^2-t^2)^2}{(z^2-1)^3}w\, dz\qquad \text{and}\qquad \eta_2=\frac{dz}{w\, (z^2-1)}.$$
Integrating and taking real parts, we have
$$
\Re\int_{z_0}^z \phi_3=\frac{c (t^2+1)}{2} \ln\left|\frac{z-1}{z+1} \right|+c (t^2-1)\,\Re\left(\frac{z}{z^2-1}\right). 
$$
showing that there are no periods on $x_3$-axis, reducing the analysis of the period to ensure that 
\begin{equation}\label{una}
\int_{\tilde{\gamma}_1^j} \eta=c^2\overline{\int_{\tilde{\gamma}_1^j} \eta_2}, \qquad
\int_{\tilde{\gamma}_2^j} \eta=c^2\overline{\int_{\tilde{\gamma}_2^j} \eta_2},
\end{equation}
where $\{\tilde{\gamma}_1^j,\tilde{\gamma}_2^j\},\ j=1,2, \ldots, k,$ is a homology basis $\{\tilde{\gamma}_1^j,\tilde{\gamma}_2^j\}$ of  $\overline{M}_k$.  This basis may be constructed as follows.  Let $\gamma_1^j$, $\gamma_2^j, \ j=1,2, \ldots, k,$ be  the oriented simple closed curves in the $z$-plane as in Figure~\ref{BH}. 
\begin{figure}[h!]
\includegraphics[totalheight=3.2cm]{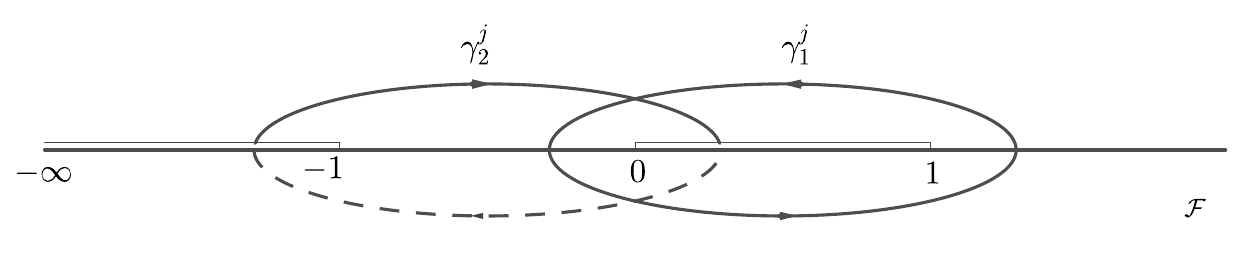}
\caption{$\mathcal{F}$-sheets on $\overline{M}_k.$}\label{BH}
\end{figure} 

The transition from dashed to continuous line on ${\gamma}_2^j$ indicates a change of sheet on $\overline{M_k}$. Let $\tilde{\gamma}_1^j$ and $\tilde{\gamma}_2^j,\ j=1,2, \ldots, k,$ be  the unique lifts of $\gamma_1^j$ and $\gamma_2^j$ to $\overline{M}_k$, respectively.

The curves $\gamma_1^j$ and $\gamma_2^j$ are homotopic to 
$\Gamma_1^j=\{z:|z|=\varepsilon\} \cup \overrightarrow{[0,1]} \cup \{z:|z-1|=\varepsilon\} \cup \overleftarrow{[0,1]}$ and 
$\Gamma_2^j=\{z:|z+1|=\varepsilon\} \cup \overrightarrow{[-1,0]} \cup \{z:|z|=\varepsilon\} \cup \overleftarrow{[-1,0]}$,  for sufficiently small  $\varepsilon >0$.   Therefore, the computation  of the integrals in \eqref{una} on $\tilde{\gamma}_1^j$ gives
\begin{equation}\label{int_G1Et}
\int_{\Gamma_1^j} \eta= \int_{|z|=\varepsilon} \eta + \int_{\overrightarrow{[0,1]}} \eta +\int_{|z-1|=\varepsilon} \eta +
\int_{\overleftarrow{[0,1]}} \eta
\end{equation}
and
\begin{equation}\label{int_G1Et2}
\int_{\Gamma_1^j} \eta_2= \int_{|z|=\varepsilon} \eta_2 + \int_{\overrightarrow{[0,1]}} \eta_2 +\int_{|z-1|=\varepsilon} \eta_2 +
\int_{\overleftarrow{[0,1]}} \eta_2.
\end{equation} 
The integrals in  \eqref{una} for $\tilde{\gamma}_2^j$ can be expressed in a similar way. In order to avoid a divergent integral, Hoffman and Meeks added an exact one-form to $\eta_2$ (see p. 17 of \cite{Hoffman.1990}),  that is,
\begin{equation}\label{eta2}
\eta_2=-\frac{dz}{2w}-\frac{k+1}{2}\,\mathrm{d}\bigg(\frac{z}{w}\bigg).
\end{equation}
In our case, we apply this technique to $\eta$. As $$\mathrm{d} w=\frac{w}{k+1}\bigg(\frac{k}{z}+\frac{2z}{z^2-1}\bigg)\,dz,$$ we obtain that
\begin{equation}\label{due}
\mathrm{d}\bigg(\frac{z w}{z^2-1}\bigg)=-\frac{2k}{(k+1)}\frac{z^2w\, dz}{(z^2-1)^2}+\frac{(2k+1)}{k+1}\frac{w\, dz}{(z^2-1)}
\end{equation}
and
\begin{equation}\label{F_eta3}
\mathrm{d}\bigg(\frac{z w}{(z^2-1)^2}\bigg)=-\frac{2 (2k+1)}{k+1}\eta_3-\frac{(2k+1)}{k+1}\frac{w\, dz}{(z^2-1)^2},
\end{equation}
where $$\eta_3:=\frac{w\, dz}{(z^2-1)^3}.$$
Also, we observe that \eqref{F_eta3} is equivalent to the following identity
\begin{equation}\label{F_eta3bis}\frac{w\, dz}{(z^2-1)^2}=-2\eta_3-\frac{k+1}{(2k+1)}\,\mathrm{d}\bigg(\frac{z w}{(z^2-1)^2}\bigg).
\end{equation}
Consequently,  from \eqref{due} and \eqref{F_eta3}, we get
$$\mathrm{d}\bigg(\frac{z w}{(z^2-1)^2}\bigg)-\frac{2k+1}{2k}\,\mathrm{d}\bigg(\frac{z w}{z^2-1}\bigg)=-\frac{2 (2k+1)}{k+1}\eta_3-\frac{(2k+1)}{2k (k+1)}\frac{w\, dz}{(z^2-1)}$$
and, thus,  we have the following expression 
\begin{equation}\label{eta3}
\eta_3=\frac{k+1}{4k}\,\mathrm{d}\bigg(\frac{z w}{z^2-1}\bigg)-\frac{w\, dz}{4k (z^2-1)}-\frac{k+1}{2 (2k+1)}\,\mathrm{d}\bigg(\frac{z w}{(z^2-1)^2}\bigg).
\end{equation}
Finally,  decomposing $\eta$ in partial fractions,  using \eqref{eta3} and \eqref{F_eta3bis} it results that
$$
\begin{aligned}
\eta&=\frac{w\, dz}{z^2-1}+2 (1-t^2)\frac{w\, dz}{(z^2-1)^2}+(1-t^2)^2\,\eta_3\\
&=\frac{w\, dz}{z^2-1}-\frac{2(k+1)(1-t^2)}{(2k+1)}\,\mathrm{d}\bigg(\frac{z w}{(z^2-1)^2}\bigg)-(1-t^2)(3+t^2)\,\eta_3
\end{aligned}
$$
and so
\begin{equation}\label{Eta_eta1}
\begin{aligned}
\eta=&\frac{(4k+3-2t^2-t^4)}{4k}\frac{w\, dz}{z^2-1}-\frac{(k+1)(1-t^2)^2}{2(2k+1)}\,\mathrm{d}\bigg(\frac{z w}{(z^2-1)^2}\bigg)+\\
&-\frac{(k+1)(1-t^2)(3+t^2)}{4k}\,\mathrm{d}\bigg(\frac{z w}{z^2-1}\bigg).
\end{aligned}
\end{equation}
Let $\omega$ to be the branch of $w = \left[z^k \, (z^2 - 1)\right]^{1/(k+1)}$, defined on $\mathbb{C} - \{(-\infty,-1]\cup [0,1] \}$, such that
$$\lim_{\varepsilon \rightarrow 0^+} \arg\Big(\omega(\frac{1}{2}-i\,\varepsilon )\Big)=-\frac{\pi i}{k+1}.$$ Hence, from \eqref{Eta_eta1} and using the computations on p.  17 of \cite{Hoffman.1990},  we conclude that the integral in the left-hand side of \eqref{una} is equal to
$$\int_{\Gamma_1^j} \eta=\frac{(4k+3-2t^2-t^4)}{4k}\int_{\Gamma_1^j} \frac{w\, dz}{z^2-1}=c_1\, A\,\frac{(4k+3-2t^2-t^4)}{4k},$$
where $$c_1=e^{-\pi i/(k+1)}-e^{\pi i/(k+1)}\qquad \text{and}\qquad A=\int_0^1\frac{[u^k\, (1-u^2)]^{1/(k+1)}}{u^2-1}\,du<0.$$
Moreover,  as
$$\int_{\Gamma_1^j} \eta_2=\overline{c_1}\, B,\qquad B=-\frac{1}{2}\int_0^1[u^k\, (1-u^2)]^{-1/(k+1)}\,du<0,$$
substituting in \eqref{una},  we conclude that we can choose 
\begin{equation}\label{const_c}
c(k,t)=\sqrt{\frac{4k+3-2t^2-t^4}{4k}}\, \sqrt{\frac{A}{B}}, \qquad |t|<\sqrt{2\sqrt{k+1}-1}.
\end{equation}

To describe the symmetries of $\Sigma_{k,t}$, we start determining the expressions of the differential $dh=g\cdot\eta$ and of $(dg\cdot \eta)$ as
$$\begin{aligned}
dh&=c\, \frac{(z^2 - t^2)}{(z^2-1)^2}\, dz,\\
dg\cdot \eta&=c\,\Bigg[\frac{2(1-t^2)\, z}{(z^2-1)^3}-\Big(\frac{k}{k+1}\Big)\, \frac{(z^2-t^2)}{z \,(z^2-1)^2}-\Big(\frac{2}{k+1}\Big)\, \frac{(z^2-t^2)\,z}{(z^2-1)^3}\Bigg]\,dz^2.
\end{aligned}$$ 
In Table~\ref{table4} we summarize the behavior of $g$, $dh$ and  $dg\cdot \eta$ along each path $\sigma_j$, $j=1,2,3$.
\begin{table}[h!]\caption{}
\renewcommand*{\arraystretch}{1.7}
\centering
\begin{tabular}{|c|c|c|c|}
 \hline
 \quad \text{Path}\; $\sigma_j$ & \quad $g \circ \sigma_j$  \quad & \quad $dh(\sigma_j')$\quad  & \quad $(dg\cdot \eta)(\sigma_j')$ \quad \\
\hline
 $\sigma_1(u)=u, \ 1< u<\infty$  & $ \r$ & $\r$ & $\r$  \\ 
\hline 
 $\sigma_2(u)=u, \ 0< u <1$  & $e^{-\frac{\pi i}{(k+1)}}\,\r$ & $\r$ & $\r$  \\  
 \hline 
$\sigma_3(u)=i\,u, \ 0< u<\infty$ & $\;e^{-\frac{(k+2)\pi i}{2 (k+1)}}\,\r\;$ & $i \r$ & $i\r$  \\ 
 \hline
\end{tabular}
\label{table4}
\end{table}

Now we study the real part of the period vector of not homotopically trivial closed curves around each puncture of $M_k$, namely $\mathfrak{p}_{-1}$,  $\mathfrak{p}_1$ and $\mathfrak{p}_{\infty}.$ According to Table~\ref{table4} and 
Proposition~\ref{Karcher.1989},  the curves $X \circ \sigma_j , j = 1,2$,  are planar geodesics  and  $X \circ \sigma_3$ is a straight line on the fundamental piece of $\Sigma_{k,t}.$ Thus,  the planes $P_1:= (x_1, x_3)$-plane and $P_2= -L_{\theta} \cdot P_1$ are reflective planes of symmetry containing the curves $X \circ \sigma_j , j = 1,2$,  respectively.  

Therefore,  a small curve $\sigma$ on $M_k$ around $\mathfrak{p}_{-1}$ is, after homology, invariant under reflections in $\sigma_j , j = 1,2$.  By virtue of this,  the period vector 
$\Re \int_{\sigma}(\phi_1, \phi_2, \phi_3)$ is perpendicular to the planes $P_1$ and $P_2$. But these planes are not parallel, because $P_1$ and $P_2$ make an angle of $\theta=\pi/(k+1)$.  Since the period vector $\Re \int_{\sigma} (\phi_1, \phi_2, \phi_3)$ must be perpendicular to both planes (see Lemma~$2$ in \cite{Wohlgemuth.1997}), we conclude that it is zero. We can apply the same argument as above to conclude that the period vector will be also zero in the other cases. \\

Let $\kappa(z,w)=(\overline{z},\overline{w})$ and $\lambda(z,w)=\left(-z, e^{i\,k\,\theta} w\right)$,  with $\theta = \pi/(k+1)$,  be the conformal mappings of $\mathbb{C}_\infty^2$, as defined in \cite{Hoffman.1990} (see p. 13).  Due to the symmetries of $M_k$ (see Corollary 3.2 in \cite{Hoffman.1990}) the group generated by $\kappa$ and $\lambda$ is the dihedral group $\mathcal{D}(2k+2)$ with $4(k+1)$ elements. By a straightforward calculation, we deduce 

$$
\left\{
\begin{aligned}
\lambda^{*}\phi_1&=\cos \theta \, \phi_1 - \sin \theta \, \phi_2,   \\
\lambda^{*}\phi_2&=\sin \theta \, \phi_1 + \cos \theta \, \phi_2,  \\
\lambda^{*}\phi_3&= - \phi_3,
\end{aligned}
\right.
\qquad \qquad 
\left\{
\begin{aligned}
\kappa^{*}\phi_1&= \overline{\phi_1}, \\
\kappa^{*}\phi_2&= - \overline{\phi_2},  \\
\kappa^{*}\phi_3&=  \overline{\phi_3}.
\end{aligned}
\right.
$$

Thus, $X \circ \lambda = L_\theta\,  X^t $ and $X \circ \kappa = K\,  X^t,$ where $ L_\theta$ and $K$ are the real orthogonal matrices defined in \eqref{matrix}.
Therefore, the group $\mathcal{D}(2k+2)=\langle L_\theta^j\, K^\ell \rangle, \ j=0, \ldots , (2k+1)$,  $\ell=0,1$,  acts by isometries on $\Sigma_{k,t}$,  which can be decomposed in $4(k+1)$ congruent pieces. See $X(\Omega)$ in Figure~\ref{FR}.  
In fact, using similar arguments as in the proof of Proposition 4.5 from \cite{Hoffman.1990}, we have that $\mathcal{D}(2k + 2)$ is the symmetry group of $\Sigma_{k,t}=X(M_k)$.

\begin{figure}[h!]
\includegraphics[totalheight=6.7cm]{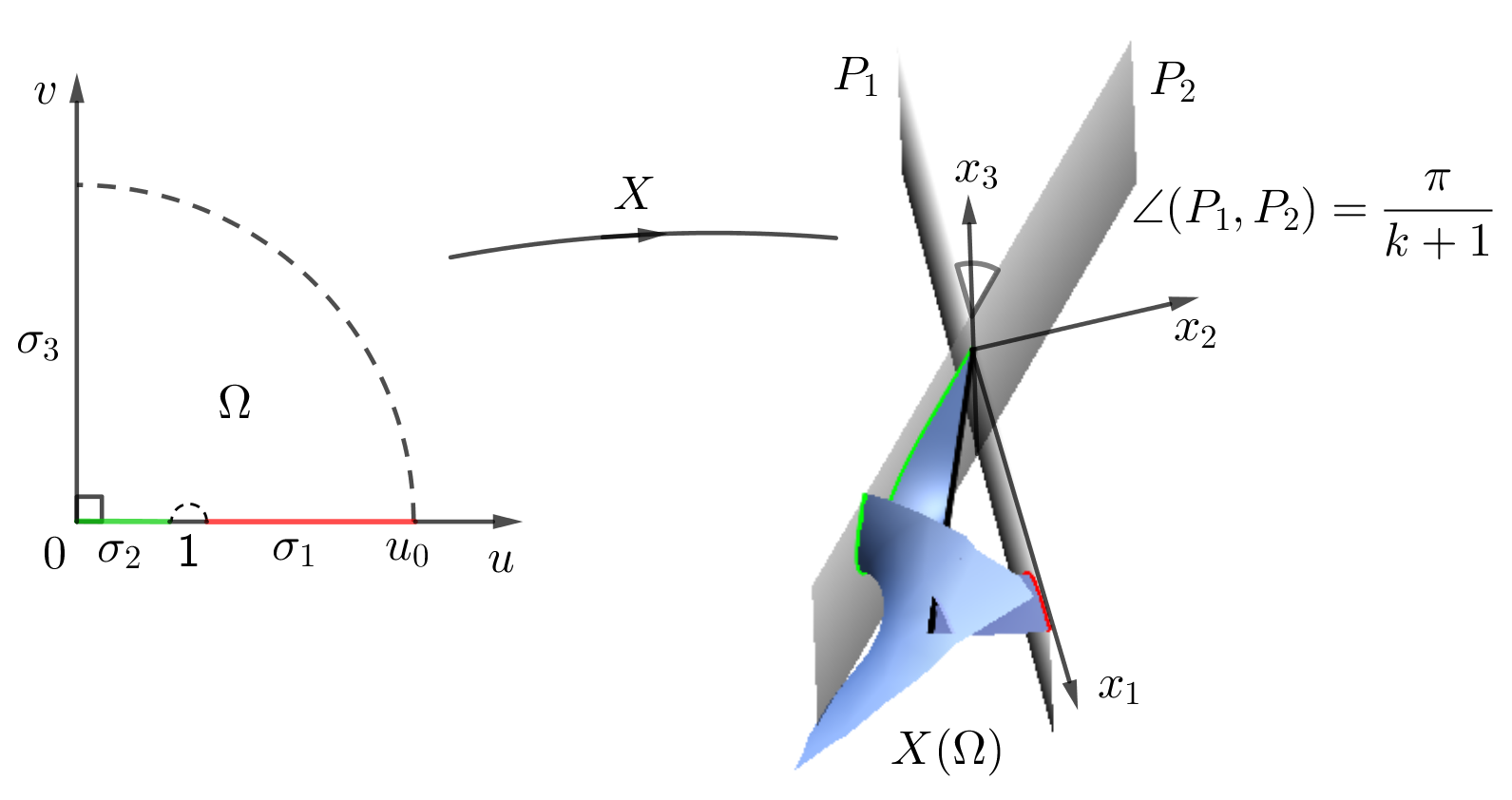}
\caption{Fundamental piece.}
\label{FR}
\end{figure}

Then,  we may conclude that the vertical planes, containing   the $x_3$-axis,  given by $P_{j+1}= -L_\theta\cdot P_j$,  $j=1,\dots, k$,  are planes of symmetry for the surfaces $\Sigma_{k,t}$ which also have $(k+1)$ straight lines of symmetry.  The regularity and the completeness of $\Sigma_{k,t}$ are easy to see.  This completes the proof of the theorem.
\end{proof}

\begin{remark}
For each $t$, the surfaces $\Sigma_{1,t}$ is $S_x^1$, given in Theorem~\ref{teo3}. In particular,  $\Sigma_{1,1}$ is the Costa surface.
\end{remark}

Making use of  \eqref{22}  and  \eqref{const_c},  a straightforward computation shows that
$$c(k,t)=2^{\frac{1}{k+1}}\sqrt{\frac{4k+3 - 2t^2-t^4}{4k}}\;\frac{\Gamma\left(\frac{k + 2}{2k + 2}\right)}{\Gamma\left(\frac{2k+3}{2k + 2}\right)}\,\sqrt{\frac{1}{2k+2}\cot\Big(\frac{\pi}{2k+2}\Big)},$$
where $\Gamma$ is the standard Gamma function. Furthermore,
\begin{equation}\label{g_inf}
\lim_{k \rightarrow \infty } c(k,t) =1,  \quad \lim_{k \rightarrow \infty } w(z) =z, \quad  \lim_{k \rightarrow \infty } g(z) = \frac{z^2-1}{z(z^2-t^2)}, \quad \lim_{k \rightarrow \infty } \eta(z) = \frac{z(z^2-t^2)^2}{(z^2-1)^3}.
\end{equation}

Thus,  if $M_k$ is scaled so that the maximum value of the Gauss
curvature is equal to $-1$ at the origin and $|t|=1$, then we have the  Weierstrass data of the Scherk's Fifth surface (see Figure~\ref{Figure_b1}) 
$$(g,\eta) = \Big(\frac{1}{z},\lambda\, \frac{z}{z^2-1}\, dz \Big), \quad \lambda \in \mathbb{R}.$$ 

Up to a suitable motion in $\mathbb{R}^3$, one can show that this pair $(g,\eta)$
produces the same minimal surface as $\cos x_2+\sinh x_1\,  \sinh x_3=0$.  Thus, we reobtain one of the limits obtained by D. Hoffman and W.H. Meeks III,  according to Theorem 3.1  in \cite{Hoffman_Meeks.1990}.\\

Now, if $|t|\neq 1$ we obtain the Weierstrass data for the family of the  Scherk-Enneper surfaces (see Figure~\ref{Figure_a1}),  which indicates that this is a possible limit, when properly normalized. Still with $|t|\neq 1$, it is not know to us whether, after a rescaling in $M_k$ and taking $k \rightarrow \infty$,  we obtain convergence to the singular configuration given by a vertical double Enneper surface of higher dihedral symmetry (see \cite{Karcher.1989}) and a horizontal plane.

\begin{figure}[h]
\subfigure[\label{Figure_a1}]{
\includegraphics[totalheight=3.3cm]{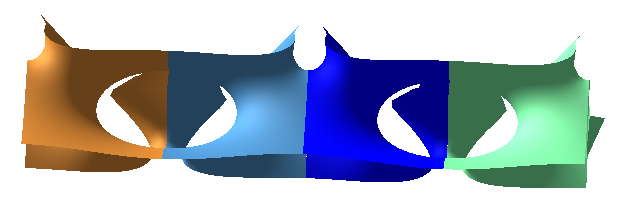}}
\subfigure[\label{Figure_b1}]{
\includegraphics[totalheight=2.8cm]{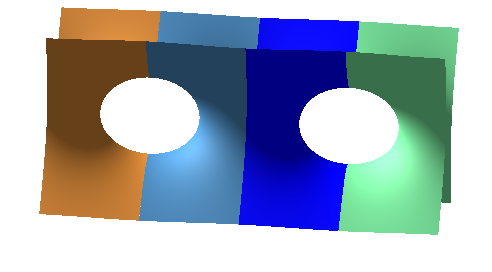}}
\caption{Computer graphics of $\Sigma_{\infty,t}$ via the data given in  \eqref{g_inf} for: (a) $t=0.8$ (Scherk-Enneper surface) and (b) $t=1$ (Scherk's Fifth surface).}
\label{Fig_Sup_Lim}
\end{figure}

\begin{remark}
In \cite{Hoffman_Meeks.1990}  D.  Hoffman and W.H. Meeks did a thorough study of limits to $\Sigma_{k,1}=M_k$. 
As $k \rightarrow \infty$, suitable scalings of the $M_k$ converge either to the singular configuration
given by a vertical catenoid and a horizontal plane passing through its waist circle, or to the singly-periodic Scherk minimal surface (see also \cite{MeeksIII.2012}). 
\end{remark}

\begin{acknowledgements}
We would like to thank  Alexandre Lymberopoulos from  University of S\~ao Paulo (Brazil) and the referee  for their careful reading of the original manuscript and comments which improved the paper.
\end{acknowledgements}

\bibliographystyle{amsplain} 
\bibliography{Bib_Minimal_Surfaces}

\providecommand{\bysame}{\leavevmode\hbox to3em{\hrulefill}\thinspace}
\providecommand{\MR}{\relax\ifhmode\unskip\space\fi MR }
\providecommand{\MRhref}[2]{%
  \href{http://www.ams.org/mathscinet-getitem?mr=#1}{#2}
}
\providecommand{\href}[2]{#2}
\begin{thebibliography}{10}

\bibitem{Ahlfors.1966}
Lars~V. Ahlfors, \emph{Complex analysis: {A}n introduction of the theory of
  analytic functions of one complex variable}, second ed., McGraw-Hill Book
  Co., New York-Toronto-London, 1966. \MR{188405}

\bibitem{Chand.1985}
K.~Chandrasekharan, \emph{Elliptic functions}, Springer-Verlag, 1985.

\bibitem{Chen.1982}
C.~C. Chen and F.~Gackstatter, \emph{Elliptische und hyperelliptische
  {F}unktionen und vollst\"{a}ndige {M}inimalfl\"{a}chen vom {E}nneperschen
  {T}yp}, Math. Ann. \textbf{259} (1982), no.~3, 359--369.

\bibitem{Costa.1982}
C.~J. Costa, \emph{Imersões mínimas em {$\mathbb{R}^3$} de gênero um e
  curvatura total finita}, Ph.D. Thesis, IMPA, Rio de Janeiro, Brasil (1982).

\bibitem{Costa.1984}
\bysame, \emph{Example of a complete minimal immersion in {$\mathbb{R}^3$} of
  genus one and three embedded ends}, Bol. Soc. Brasil. Mat. \textbf{15}
  (1984), no.~1-2, 47--54.

\bibitem{Nedir.1994}
N.~Do~Espirito-Santo, \emph{Complete minimal surfaces in {$\mathbb{R}^3$} with
  type {E}nneper end}, Ann. Inst. Fourier (Grenoble) \textbf{44} (1994), no.~2,
  525--557.

\bibitem{Fang.1990}
Y.~Fang, \emph{A new family of {E}nneper type minimal surfaces}, Proc. Amer.
  Math. Soc. \textbf{108} (1990), no.~4, 993--1000.

\bibitem{Gackstatter.1976}
F.~Gackstatter, \emph{{\"U}ber die dimension einer minimalfl{\"a}che und zur
  ungleichung von st. cohn-vossen}, Arch. Rational Mech. Anal. \textbf{61}
  (1976), no.~2, 141--152.

\bibitem{Hoffman_Meeks.1990}
D.~Hoffman and W.~H. Meeks, \emph{Limits of minimal surfaces and scherk's fifth
  surface}, Arch. Rational Mech. Anal. \textbf{111} (1990), no.~2, 181--195.

\bibitem{Hoffman.1985}
D.~Hoffman and W.~H. Meeks~III, \emph{A complete embedded minimal surface in
  {$\mathbb{R}^3$} with genus one and three ends}, J. Differential Geom.
  \textbf{21} (1985), no.~1, 109--127.

\bibitem{Hoffman.1990}
\bysame, \emph{Embedded minimal surfaces of finite topology}, Ann. of Math. (2)
  \textbf{131} (1990), no.~1, 1--34.

\bibitem{Fujimoto.2013}
David Hoffman and Hermann Karcher, \emph{Complete embedded minimal surfaces of
  finite total curvature}, Geometry, {V}, Encyclopaedia Math. Sci., vol.~90,
  Springer, Berlin, 1997, pp.~5--93. \MR{1490038}

\bibitem{Jorge.1983}
L.~P. Jorge and W.~H. Meeks~III, \emph{The topology of complete minimal
  surfaces of finite total gaussian curvature}, Topology \textbf{22} (1983),
  no.~2, 203--221.

\bibitem{Kang.2003}
J.~L. Kang and H.~Wang, \emph{A family of complete immersed minimal surfaces
  with only one end}, J. Zhejiang Univ. Sci. Ed. \textbf{30} (2003), no.~6,
  612--616.

\bibitem{Karcher.1989}
H.~Karcher, \emph{Construction of minimal surfaces, surveys in geometry},
  University of Tokyo, 1989.

\bibitem{Batista.2006}
F.~Mart\'in and V.~Ramos~Batista, \emph{The embedded singly periodic
  {S}cherk-{C}osta surfaces}, Math. Ann. \textbf{336} (2006), no.~1, 155--189.
  \MR{2242622}

\bibitem{MeeksIII.2012}
William~H. Meeks, III and Joaqu\'{\i}n P\'{e}rez, \emph{A survey on classical
  minimal surface theory}, University Lecture Series, vol.~60, American
  Mathematical Society, Providence, RI, 2012. \MR{3012474}

\bibitem{Osserman.2013}
R.~Osserman, \emph{A survey of minimal surfaces}, Courier Corporation, 2013.

\bibitem{Batista.2004}
V.~Ramos~Batista, \emph{Noncongruent minimal surfaces with the same symmetries
  and conformal structure}, Tohoku Math. J. (2) \textbf{56} (2004), no.~2,
  237--254. \MR{2053320}

\bibitem{Thayer.1995}
E.~C. Thayer, \emph{Higher-genus {C}hen-{G}ackstatter surfaces and the
  {W}eierstrass representation for surfaces of infinite genus}, Experiment.
  Math. \textbf{4} (1995), no.~1, 19--39.

\bibitem{vilhena.2021}
Jos\'{e} Antonio~M. Vilhena, \emph{A family of genus one minimal surfaces with
  two catenoid ends and one {E}nneper end}, Differential Geom. Appl.
  \textbf{77} (2021), Paper No. 101766, 15.

\bibitem{Wohlgemuth.1991}
M.~Wohlgemuth, \emph{Higher genus minimal surfaces by growing handles out of a
  catenoid}, Manuscripta Math. \textbf{70} (1991), no.~4, 397--428.

\bibitem{Wohlgemuth.1997}
\bysame, \emph{Minimal surfaces of higher genus with finite total curvature},
  Arch. Rational Mech. Anal. \textbf{137} (1997), no.~1, 1--25.

\end{thebibliography}

\end{document}